\documentclass{amsproc}

\usepackage{fullpage}
\usepackage{psfrag}
\usepackage{setspace}
\usepackage{paralist}

\usepackage{multicol}
\usepackage{color}
\usepackage{array,amscd,amsmath,amsthm,ifthen}
\usepackage{amssymb}
\usepackage{graphicx}
\usepackage[all]{xy}
\usepackage{multicol}
\usepackage{cite}

\def\M{\mathcal{M}}

\def\C{\mathcal{C}}
\def\G{\mathcal{G}}

\def\L{\mathcal{L}}
\def\cO{\mathcal{O}}

\def\V{\mathcal{V}}

\def\K{\mathcal{K}}
\def\Y{\mathcal{Y}}

\def\Z{\mathcal{Z}}
\def\P{\mathcal{P}}
\def\X{\mathcal{X}}

\def\PP{\mathbb{P}}

\def\codim{\mathrm{codim}}

\newtheorem{thm}{Theorem}[section]
\newtheorem{prop}[thm]{Proposition}
\newtheorem{lemma}[thm]{Lemma}

\theoremstyle{definition}

\numberwithin{equation}{section}

\newcommand{\be}{\begin{equation}}
\newcommand{\ee}{\end{equation}}

\begin{document}

\title {Two remarks on the Weierstrass flag}

\author{Enrico Arbarello}
\address{Universit\`a di Roma ``La Sapienza'' - Dipartimento di Matematica
- Piazzale Aldo Moro 5, 00185 Roma - Italy}
\email{ea@mat.uniroma1.it}

\author{Gabriele Mondello}
\address{Universit\`a di Roma ``La Sapienza'' - Dipartimento di Matematica
- Piazzale Aldo Moro 5, 00185 Roma - Italy}
\email{mondello@mat.uniroma1.it}


\date{August 16, 2011}

\begin{abstract}
We show that the locally closed strata of the Weierstrass flags
on $M_g$ and $M_{g,1}$ are almost never affine.
\end{abstract}

\maketitle

\begin{section}{Introduction}

\hskip 0.4cm
The coarse moduli space $M_{g,n}$ of curves of genus $g$ with $n$
marked points is a quasi-projective variety.
Grothendieck \cite{grothendieck:esquisse} already wondered how many
affines are needed to cover $M_{g,n}$. A first hint came from
Diaz's upper bound \cite{diaz:complete} on the dimension of a complete
subvariety of $M_{g,n}$, then strengthened by Harer's computation
\cite{harer:virtual} of the virtual cohomological dimension of the
mapping class group. Looijenga's vanishing \cite{looijenga:tautological}
of the tautological classes
in high degree motivated the following.

\vspace{0.3cm}
{\it{\bf{Question} (Looijenga):}
Does $M_g$ have an affine stratification with $g-1$ layers?
Does $M_{g,1}$ have an affine stratification with $g$ layers?}
\vspace{0.3cm}

In both cases, stratifications with the right number of layers
do exist and it is natural to ask whether the layers are affine.\\

In this short paper, we will concentrate
on the Weierstrass flag studied by Arbarello \cite{arbarello:weierstrass},
whose strata had already been defined by Rauch
\cite{rauch:weierstrass}.
As Mumford pointed out in \cite{mumford:towards},
Section 7, the proof of Theorem (3.27) in \cite{arbarello:weierstrass}
is incomplete. Therefore, it is still not known whether one may exclude that any single stratum 
of this flag contains compact curves.
Actually, one may even ask whether  these strata are affine.\\

In this note, we will show that almost no such stratum is affine.\\

Clearly, these negative results about
a specific stratification
do not conflict with Looijenga's question.
In fact, for $g\leq 5$,
Fontanari-Looijenga \cite{fontanari-looijenga:perfect} show that
a good stratification exists and Fontanari-Pascolutti
\cite{fontanari-pascolutti:cover} exhibit a good affine cover of $M_g$.

\begin{subsection}{Content of the paper}
Proposition \ref{prop.1} deals with the Weierstrass stratification
in $M_{g,1}$. The idea is to show that the strata can be realized as
open subsets of smooth varieties, whose complement
(representing a certain family of plane curves)
is not purely divisorial.
The techniques are borrowed from Arbarello-Cornalba
\cite{arbarello-cornalba:petri} (see also
Chapter XXI of \cite{ACG:II}). Similar computations can be also
found in Caporaso-Harris \cite{caporaso-harris:plane}.

Proposition \ref{prop.2} deals with the stratification of $M_g$
and relies on the same idea.
The key computation is borrowed from Diaz \cite{diaz:two}.\\

We work over the field of complex numbers, but all the results
hold over an algebraically closed field of characteristic zero.\\
\end{subsection}
\begin{subsection}{Acknowledgments}
We thank Eduard Looijenga for a continuous exchange of ideas 
on this topic and the referee for useful suggestions.
\end{subsection}
\end{section}
\begin{section}{Linear series of Weierstrass type}

\hskip 0.4cm

Below we describe a slight variation of  the theory of deformation of $g^1_d$'s
and $g^2_d$'s on smooth curves. That theory is described in
Chapter XXI of \cite{ACG:II}. 
Set
$$
\G^r_{d,*}=\{[(C,p,Z)] \,\,|\,\, 1\in Z\subset H^0(C,dp),\,\, \dim(Z)=r+1\}
$$
which naturally sits inside the coarse
space associated to $\G^r_d\times_{\M_g} \M_{g,1}$.

Following the same  arguments to prove the smoothness and to compute the dimension of 
$\G^1_d$ as
in \cite{ACG:II}, Chapter XXI, Proposition 6.8, one proves that
 the variety $\G^1_{d,*}$ is smooth   and of dimension  equal to $2g+d-3$.

\begin{subsection}{Plane curves}
Let us  next recall the basic setting for the study of $\G^2_d$ as described in Sections 8, 9 and 10, Chapter XXI of \cite{ACG:II}.

 Let $C$ be a smooth genus $g$  curve and  $\varphi: C\to\PP^2$   a nonconstant
 morphism. The normal sheaf $N_\varphi$ to this morphism
 is defined by the exact sequence
 $$
 0\to T_C\overset{d\varphi}\to\varphi^*T_{\PP^2}\to N_\varphi\to0\, .
 $$
 Let
\be
 \xymatrix{
\C \ar[d]_{\pi} \ar[r]^{\widetilde\varphi} & \PP^2 
\\
(U,u_0) 
}\label{diagram1}
\ee
 be a deformation of $\varphi$ parametrized by a pointed  analytic space $(U, u_0)$,  so that
$\iota:C\overset{\cong}{\hookrightarrow}\pi^{-1}(u_0)$
 and $\widetilde\varphi\circ\iota=\varphi$.
The characteristic homomorphism of this family is the homomorphism
\be
T_{u_0}(U)\to H^0(C,N_\varphi)\label{horikawa1}
\ee
where $T_{u_0}(U)$ is the Zariski tangent space to $U$ at $u_0$,
assigning to each tangent vector to $U$ at $u_0$ the Horikawa class of the corresponding infinitesimal deformation of $\varphi$. Denote by $N'_\varphi$ the line bundle  quotient of $N_\varphi$, i.e $N'_\varphi=N_\varphi/\text{\it{Torsion}}$.
The line bundle $N'_\varphi$
 may also be defined by the exact sequence
$$
 0\to T_C(R)\overset{d\varphi}\to\varphi^*T_{\PP^2}\to N'_\varphi\to0
 $$
where  $R$ be the ramification divisor of $\varphi$. In Proposition 9.10 (loc.cit)
it is proved that, if the restriction of $\widetilde\varphi$ to each fiber of $\pi$ is {\it birational}, then
 for a general $u\in U$ the image of the characteristic homomorphism
$T_{u}(U)\to H^0(C_u,N_{\varphi_u})$ does not intersect the kernel
of $H^0(C_u,N_{\varphi_u})\to H^0(C_u,N'_{\varphi_u})$,
 where $C_u=\pi^{-1}(u)$
and $\varphi_u=\widetilde\varphi|_{C_u}$ (here by general point of $U$  we mean a general point of one of its irreducible components).

\vskip 0.3cm

A local universal deformation for a morphism $\varphi: C\to \PP^2$
can be constructed as follows, at least when $g\geq 2$ (the cases $g=0,1$ are best treated separately).
Let  $d=\deg \varphi^*(\cO_{\PP^2}(1))$, let $\K\to B$ be a Kuranishi family for $C$. 
Consider  the Brill-Noether variety $\G^2_d$ over $B$ and let 
$\V$ be the bundle of projective frames for the universal $g^2_d$ over $\G^2_d$.
Pulling back $\K$ to $\V$, yields a deformation of $\varphi$

\be
 \xymatrix{
\K \ar[d] & \ar[l] \K\times_B\V=\C \ar[d]_{\Pi} \ar[rr]^{\qquad\Phi} && \PP^2\\ 
B & \ar[l] (\V,v_0) 
}\label{diagram2}
\ee
 parametrized by $\V$, with $C\cong\Pi^{-1}(v_0)$ and $\varphi\cong \varphi_{v_0}=\Phi|_{C_{v_0}}$.
Moreover, for a general point
$v\in \V$, the characteristic homomorphism yields an isomorphism
$$
T_v(\V)\cong H^0(C_v,N_{\varphi_v})\,.
$$

In Theorem 10.1,  Chapter XXI  (loc.cit) it is proved that if $\X$ is an irreducible component of $\G^2_d$ whose general point corresponds to a curve $C$ of genus $g$ equipped with a basepoint-free $g^2_d$, which {\it is not} composed with an involution,  then $\dim \X=3d+g-9$,
or equivalently $\dim \V'=3d+g-1$
where $\V':=\V\times_{\G^2_d}\X$ is the pull-back to $\X$ of the bundle $\V$.
There the theorem is proved under the assumption that $g\geq 2$, but the cases $g=0,1$ can be treated in a similar way. 

In Theorem 10.14,  Chapter XXI (loc.cit)  it is proved that, for every (non-negative) value of $d$ and $g$ such that
$(d-1)(d-2)/2\geq g$, there exists a genus $g$ curve $C$  equipped with a basepoint-free $g^2_d$ which realizes $C$ as a plane nodal curve of degree $d$. 
\end{subsection}
\begin{subsection}{Plane curves with a total tangency point}
%
In order to study $\G^2_{d,*}$ we need to consider the appropriate
deformation problem. Let us fix a point $Q$ and a line $L$ in $\PP^2$. Look at the deformation (\ref{diagram1}) and suppose that 
\vskip 0.2 cm
\begin{itemize}
\item[(a)]
$\pi: \C \to U$ is a family of pointed curves i.e.  
$\pi$ has a section $\sigma$. 
\item[(b)]
For each $u\in U$, $\varphi_u(\sigma(u))=Q$. 
\item[(c)]
For each $u\in U$, $\varphi_u$ is birational.
{\item[(d)]
The plane curve $\Gamma_u=\varphi_u(C_u)$  is unibranched at $ Q$ with tangent line $L$ intersecting $\Gamma_u$ in  $Q$ with multiplicity $d=\deg\Gamma_u$.}
\end{itemize}

Imitating the arguments in (loc.cit) one sees that
the characteristic homomorphism (\ref{horikawa1}) factors through the inclusion $H^0(C ,N_{\varphi}(-dp))\subset H^0(C,N_{\varphi})$,
 where $p=\sigma(u_0)$ and that, moreover,  for a general $u\in U$ the image of the characteristic homomorphism
does not intersect the kernel
of $H^0(C,N_{\varphi}(-dp))\to H^0(C,N'_{\varphi}(-dp))$.

\vskip 0.3cm

Assume now $g\geq 2$; the cases $g=0,1$ can be easily dealt with separately.

Consider the natural morphism $\tau: \G^2_{d,*}\to\G^2_{d}$,
which is finite-to-one,
and its restriction to the irreducible component $\X$. 
An element of $\V'$
corresponds to a triple $(C, p, \varphi)$, where $\varphi: C\to \PP^2$ is the morphism associated to a frame of a subspace $Z\subset H^0(C, dp)$, with $[(C,p,Z)]\in \X$. Fix  a point $Q$ and a line $L$ in $\PP^2$. Let $\V^*$ be the subbundle of $\V'$ given by those frames having the property  that the corresponding morphism $\varphi: C\to \PP^2$ is such that $\varphi(p)=Q$
and $\varphi^*(L)=dp$. Hence, if $\C^*$ is the restriction of the family
$\C\rightarrow\V$ over $\V^*$, the family
\be
 \xymatrix{
\ \C^* \ar[d]_{\Pi^*} \ar[r] & \PP^2\\ 
(\V^*,v) 
}
\ee
satisfies
conditions (a), (b) and (d) above and it
is a local universal deformation.

\vskip 0.3cm

Theorem 10.1,  Chapter XXI  (loc.cit), in the present situation, translates into
the following.
 
 \begin{lemma}\label{dimension} Let  $\X$ be an irreducible component of $\G^2_{d,*}$ whose general point corresponds to
  a triple $(C,p, Z)$, where $(C, p)$  is a genus $g$ pointed curve and $Z$ is a three-dimensional 
  subspace of $H^0(C, dp)$, with  $1\in Z$ and whose corresponding 
$g^2_d\subset |dp|$ is  basepoint-free  and not composed with an involution. Then $\dim \X=2d+g-6$.
 \end{lemma}
 
\begin{proof}
Assume that $g\geq 2$. The cases $g=0,1$ can be easily treated separately.
 For our purposes we may restrict our attention to a small neighbourhood 
of $[(C,p,Z)]$ in $\X$.
By assumption, 
$h^0(dp)=l>2$ and the linear series $|dp|$ is fixed point free and not composed with an involution;
so, the family $\Pi^*$ satisfies also condition (c) above.

Clearly,
\be
\dim\V^*=\dim \X+5\, .
\label{dimvstar1}
\ee
Since, for a plane curve, imposing a $d$-fold tangency with a given line $L$ at given point $Q$
amounts to $d$ linear conditions, we also have
\be
\dim\V^*\geq\dim \V'-d=2d+g-1\,.
\label{dimvstar2}
\ee

 If $v$   is a general point of $\V^*$ corresponding   to a point $[(C,p,Z)]\in \X$ and a morphism $\varphi: C\to \PP^2$,  we get a commutative diagram
 \be
 \xymatrix{
 T_v(\V^*)\ar@{^{(}->}[r]\ar[d]_\alpha & T_v(\V')\ar[d]^{\cong}\\
H^0(C, N_\varphi(-dp)) \ar@{^{(}->}[r]& H^0(C, N_\varphi)
}\, .
\ee
Since the image of $\alpha$ does not intersect the kernel of 
$H^0(C, N_\varphi(-dp))\to H^0(C, N'_\varphi(-dp))$, we get
$$
\dim \V^*\leq h^0(N'_\varphi(-dp))\, .
$$
From (\ref{dimvstar2}) it   follows that the line bundle $ N'_\varphi(-dp)$ is non-special so that
 $H^1(C, N'_\varphi(-dp)) =H^1(C, N_\varphi(-dp)) =0$. On the other hand, 
 via the Euler sequence, we get
$$
 N'_\varphi=\omega_C(-R)\otimes \phi^*\cO_{\PP^2}(3)
$$
By Riemann-Roch, we get
\begin{align*}
\dim \V^* & \leq (\mathrm{deg}(N'_\varphi)-d)-g+1=
2g-2-\mathrm{deg}(R)+3d-d-g+1\\
& = g+2d-1-\mathrm{deg}(R)\leq g+2d-1\, .
\end{align*}
The lemma follows now from the above inequality,
together with (\ref{dimvstar1}) and (\ref{dimvstar2}).
\end{proof}

  \begin{lemma}\label{existence} For every $g$ and $d$ such that $(d-1)(d-2)\geq 2g$ there exists a genus $g$ pointed curve $(C,p)$ equipped with a basepoint-free $g^2_d\subset |dp|$ which realizes $C$ as a plane nodal curve $\Gamma$ of degree $d$, with a smooth point (the image of  $p$)
whose tangent line
has intersection multiplicity $d$ with $\Gamma$.
 \end{lemma}
\begin{proof}
 Here again to we use a slight variation of the arguments used in Section 10
 of Chapter XXI (loc.cit). 
Fix  a point $Q$ and a line $L$
 in $\PP^2$ and denote by $\Sigma_{d,g}^*$ the continuous system of all irreducible plane curves $\Gamma$ of degree $d$ whose normalization has genus $g$,
 and such that
 $(\Gamma\cdot L)_Q=d$. Let  $\varphi: C\to\Gamma$  be  the normalization.
We consider an  irreducible component $\Sigma^*$ of $\Sigma_{d,g}^*$
having the property that its general member represents a plane curve which is smooth at $Q$.
 Then one proves that (if non-empty)
 $\Sigma^*$ has dimension equal to $2d+g-1$ and that a general point of $\Sigma^*$ corresponds to a plane irreducible  curve of degree $d$ having
 $\delta=(d-1)(d-2)/2-g$ nodes, and no other singularity. The proof of this fact is, word by word,  the proof of Theorem 10.7 (loc.cit),
where one should  substitute the normal sheaf $N_\varphi$ with  $N_\varphi(-dp)$. To prove that a non empty component $\Sigma^*$ exists, one may proceed as follows. Looking at the curve $y=x^d$ one sees that $\Sigma_{d,0}^*$ is non empty. Let $\Gamma_0$
 be a rational nodal curve corresponding to a general point of (a component) of
 $\Sigma_{d,0}^*$. Mimicking the arguments used to prove Lemma 10.15 (loc.cit),
 but again using  $N_\varphi(-dp)$ instead  of $N_\varphi$, one shows that,
given any integer $k$ with $0\leq k\leq (d-1)(d-2)/2$,
there exists a deformation of $\Gamma_0$ whose general member
is a plane irreducible curve $\Gamma$ having
 $\delta=(d-1)(d-2)/2-k$ nodes and no other singularity,
and having $Q$ as a simple point with $(\Gamma\cdot L)_Q=d$.
\end{proof}
 
\end{subsection}
 
\end{section}
\begin{section}{On the Weierstrass flags}
\hskip 0.4cm 
Let $g$ and $d$ be integer such that $2\leq d\leq g+1$.
Consider the closed subvariety
\[
\overline{W}_*(d)=\{[(C,p)]\in M_{g,1}\,\,|\,\, h^0(C,dp)\geq 2\}
\]
of $M_{g,1}$ of dimension $2g-3+d$ and notice that
\[
\overline{W}_*(2)\subset\overline{W}_*(3)\subset
\dots\subset\overline{W}_*(g)\subset\overline{W}_*(g+1)
=M_{g,1}
\]
is a stratification of $M_{g,1}$, whose locally closed strata
are given by
\[
W_*(d)=\overline{W}_*(d)\setminus\overline{W}_*(d-1)
=\{[(C,p)]\in M_g\,\,|\,\,h^0(C,dp)=2\}\,.
\]

Similarly, consider the closed subvariety of $M_g$ of dimension
$2g-3+d$
\[
\overline{W}(d)=\pi(\overline{W}_*(d))=
\{[C]\in M_g\,\,|\,\,\exists\, p\in C\,\,\text{with}\,\,h^0(C,dp)\geq 2\}
\]
where $\pi: M_{g,1}\rightarrow M_g$ is the forgetful morphism.
Then
\[
\overline{W}(2)\subset\overline{W}(3)\subset
\dots\subset\overline{W}(g)=M_{g}
\]
is a stratification of $M_{g}$, whose locally closed strata
are given by
\[
W(d)=\overline{W}(d)\setminus\overline{W}(d-1)=
\left\{
[C]\in M_g\,\,
\Big|
\,\,
\begin{array}{l}
\exists p\in C\,\,\text{with}\,\, h^0(C,dp)=2,
\,\,\text{and}\\
h^0(C,(d-1)q)=1\,\,\text{for all $q\in C$}
\end{array}
\right\}
.
\]
\bigskip

Consider first the strata $W_*(d)$ of $M_{g,1}$.

\begin{prop}\label{prop.1} 
Let $5\leq d\leq g+1$.
If $d$ is not prime, or if $d$ is a prime and $(d-1)(d-2)\geq 2g$, then
$W_*(d)$ is not affine.
\end{prop}

\begin{proof}
There is a natural forgetful morphism
\[
\xymatrix@R=0in{
\G^1_{d,*}\ar[rr] && \overline{W}_*(d)\\
[(C,p,Z)] \ar@{|->}[rr] && [(C,p)]
}
\]
that restricts to an isomorphism
\[
\G^1_{d,*}\setminus(\G^1_{d-1,*}\cup\Z)\cong W_*(d)
\]
where
\[
\Z=\{[(C,p,Z)]\in\G^1_{d,*}\,\,|\,\, 1\in Z\subset H^0(C,dp), \,\,\dim(Z)=2,\,\, h^0(dp)>2\}\,.
\]

The variety $\G^1_{d,*}$ is smooth   and of dimension  equal to $2g+d-3$.
Since $\G^1_{d-1,*}\subset \G^1_{d,*}$ is a divisor,
in order to show that $W_*(d)$ is not affine we will prove that 
\begin{itemize}
\item[(I)]
$\Z\nsubseteq \G^1_{d-1,*}$
\item[(II)]
every irreducible component $\Z'$ of $\Z$ not entirely contained
in $\G^1_{d-1,*}$ satisfies the condition:
$\codim_{\G^1_{d,*}} \Z'>1.$
\end{itemize}

\vskip 0.2 cm
First suppose that $d$ is not prime and set $d=hk$, with $h>1$ and $k>1$.
 Let $C$ be a genus $g$,
 $k$-sheeted cover of $\PP^1$ with a point $p$ of total ramification.
 Since $d=hk$, with $h>1$, we may choose an element
$f\in H^0(C,dp)\smallsetminus H^0(C,(d-1)p)$ and set
$Z=\langle 1,f\rangle $. Then $[(C,p, Z)]\in \Z\smallsetminus \Z\cap \G^1_{d-1,*}$. This proves
(I). We next turn our attention  to point (II).
\vskip 0.2 cm
Let  
$\Z'$ be an irreducible component of $\Z$ not entirely contained in $\G^1_{d-1,*}$. We can assume that a general point of $\Z'$ corresponds to a triple $(C,p, Z)$, as above, with 
$h^0(dp)=l>2$ and 
$h^0((d-1)p)=l-1$ (no fixed points for $|dp|$).
\vskip 0.2 cm

Two cases may occur:
\begin{itemize}
\item[(i)]
At a general point of $\Z'$,
the linear series $|dp|$ is composed with an involution.
\item[(ii)]
At a general point of $\Z'$,
the linear series $|dp|$ is not composed with an involution.
\end{itemize}

\vskip 0.2 cm

{\bf{Case (i).}}

The curve $C$ is a $\nu$-sheeted cover of curve $\Gamma$,
with $d=k\nu$, and with a point of total ramification. Let $F:C\to \Gamma$ be this cover and let
 $\gamma$ be the genus of $\Gamma$.
The number $\delta$ of branch points of $F$ (including $F(p)$)
is at most
$$
\delta\leq 2g-2\nu(\gamma-1)-\nu\, .
$$
Therefore,
 $$
\aligned
\dim\Z'&\leq \delta+(2\gamma+k-3)\\
& \leq 2g+(2-2\nu)(\gamma-1)-\nu-k-1,\qquad&&\text{if}\quad \gamma\geq1,\\
\dim\Z'&\leq [(\delta-1)-2]+[(k+1)-2]\leq
2g+\nu+k-4,\qquad&&\text{if}\quad \gamma=0.\\\\
\endaligned
$$
In all cases, using
the fact that
$d=k\nu$ and $k,\nu>1$, one sees that
$$
\dim\Z'\leq d+2g-5=\dim \G^1_{d,*}-2\,.
$$
is always satisfied except for $d<5$, $k=\nu=2$ and $\gamma=0$.
Hence, the inequality always holds if $d\geq 5$.

 \vskip 0.5 cm 

{\bf{Case (ii).}}

Let $[(C,p,Z)]$ be a general point of $\Z'$.
By assumption 
$h^0(dp)=l>2$ and the linear series $|dp|$ is fixed point free and not composed with an involution. There is a natural $\PP^{l-3}$-bundle $\P$ over $\Z'$ whose fiber over
$[(C,p,Z)]$ is $\PP(H^0(C,dp)/Z)$. A point of this bundle over $[(C,p,Z)]$
can be viewed as a triple $[(C,p,P)]$, where $P$ is a 3-dimensional subspace of  $H^0(C,dp)$ containing
$Z$. Thus, there is an injective map
$$
\P\to\G^2_{d,*}\,.
$$
Therefore, by Lemma \ref{dimension},
$$
\dim \Z'=\dim\P-(l-3)\leq 2d+g-3-l\, .
$$
As $l\geq 3$ and $d\leq g+1$, we obtain
$2d+g-3-l\leq 2g+d-5$, and so
the codimension  of $\Z'$ in 
$\G^1_{d,*}$ is strictly greater than $1$, proving the first part of the proposition.
\vskip 0.3 cm
Suppose next that $(d-1)(d-2)/2-g\geq0$. Then, by Lemma \ref{existence},
there exists an irreducible component $\Z'$ of $\Z$ which falls in case (ii)
and again we are done.

\end{proof}

\bigskip

Now we turn our attention to the locally closed
strata $W(d)$ of $M_g$.

\begin{prop} \label{prop.2}Let $g\geq 6$
and $5\leq d\leq g-1$.
Then 
$W(d)$ is not affine.
\end{prop}

\begin{proof}
The forgerful map $\pi$ restricts to
$$
\pi|_{\overline{W}_*(d)}: \overline{W}_*(d)\to \overline{W}(d)\,
$$ 
and, in particular, to
$$
\pi|_{W_*(d)\smallsetminus  \Y}: W_*(d)\smallsetminus  \Y \to W(d) 
$$
which is finite and surjective, where
$$
\Y=\{[(C,p)]\in W_*(d)\,\,|\,\,\exists\,q\neq p,
\,\,\text{with}\,\,h^0(C,(d-1)q)\geq 2\}\, .
$$
Hence, $W(d)$ is affine only if
$W_*(d)\smallsetminus \Y$ is.
Since $W_*(d)$ is smooth, to show that $W(d)$
is affine it suffices to find a non-empty component $\Y'$ of $\Y$ of
codimension greater than $1$ inside $W_*(d)$.
This is a direct consequence of Theorem  3.2 in Diaz's
paper \cite{diaz:two}.
There the following is proven. 
\vskip 0.2 cm

\begin{thm}[\cite{diaz:two}]
Let $g\geq 4$ and let $k\leq l\leq g-1$.
If $k, l\geq\tfrac{1}{2}(g+2)$, then there exists a non-empty component $W(k,l)$
of the locus of points
in $M_g$ corresponding to curves possessing both a Weierstrass point 
of type $k$ and a Weierstrass point of type $l$ which has dimension 
$g-3+k+l$.
\end{thm}

\vskip 0.2 cm

A closer inspection of Diaz' proof also shows that:

\vskip 0.2cm

{\it For $g\geq 4$ and $k\leq l\leq g-1$,
if $k<\tfrac{1}{2}(g+2)$, then there exists a non-empty component
$W(k,l)$ of the locus of points in $M_g$ corresponding to curves
possessing both a Weierstrass point of type $k$ and a Weierstrass point
of type $l$ which has dimension at most $2g-1+l-k$.}

\vskip 0.2cm

By a Weierstrass point of type $h$ on a curve $C$ it is a meant a point $p\in C$ for which
$h^0(C, h p)\geq 2$. In proving this assertion, Diaz also shows that, if $k<l$, then
for a general $[C]\in W(k,l)$
there exists a point $p\in C$ with $h^0(C, lp)=2$. We can then take 
$$
\Y'=\pi^{-1}(W(d-1,d))\cap W_*(d)\,\, .
$$
Hence, for $4+g/2\leq d\leq g-1$,
\[
\dim\Y'
=g-3+(d-1)+d=g+2d-4\leq d+2g-5=\dim\, W_*(d)-2
\]
On the other hand, for $5\leq d< 4+g/2$,
\[
\dim\Y'
\leq 2g-1+d-(d-1)=2g\leq d+2g-5=\dim\, W_*(d)-2
\, .
\]
\end{proof}
\end{section}

\bibliographystyle{amsplain}
\bibliography{Weierstrass_flag_revised}

\end{document}